\documentclass[12pt]{article}

\usepackage[utf8]{inputenc}
\usepackage[margin=1in]{geometry}
\usepackage{amsmath}
\usepackage{amssymb}
\usepackage{amsthm}
\usepackage{mathtools}
\usepackage[english]{babel}
\usepackage{fancyhdr}
\usepackage{titling}
\usepackage[colorlinks=true, allcolors=blue]{hyperref}

\newtheorem{theorem}{Theorem}[section]
\newtheorem{proposition}[theorem]{Proposition}
\newtheorem{lemma}[theorem]{Lemma}
\theoremstyle{definition}
\newtheorem{definition}[theorem]{Definition}

\title{Minimal Nilpotent Orbits of type $G_2$, $F_4$ and $E_8$}
\author{Boming Jia}
\date{}
\hypersetup{
    pdftitle={Minimal Nilpotent Orbits of type G2, F4 and E8},
    pdfauthor={Boming Jia},
    pdfsubject={Minimal Nilpotent Orbits of type G2, F4 and E8}
}

\DeclareMathOperator{\Spec}{Spec}
\DeclareMathOperator{\Sym}{Sym}
\DeclareMathOperator{\Ad}{Ad}
\DeclareMathOperator{\ad}{ad}
\DeclareMathOperator{\rk}{rank}
\DeclareMathOperator{\Aut}{Aut}
\DeclareMathOperator{\Stab}{Stab}
\DeclareMathOperator{\gr}{gr}
\newcommand{\C}{\mathbb C}
\newcommand{\Gm}{\mathbb G_m}
\newcommand{\SL}{\mathrm{SL}}
\newcommand{\res}{\mathrm{res}}
\newcommand{\GL}{\mathrm{GL}}
\newcommand{\PGL}{\mathrm{PGL}}
\newcommand{\PP}{\mathbb P}
\newcommand{\so}{\mathfrak{so}}
\newcommand{\Lieg}{\mathfrak{g}}
\newcommand{\h}{\mathfrak h}
\newcommand{\OO}{\mathcal O}
\newcommand{\Omin}{\mathcal O_{\min}}
\newcommand{\Ominbar}{\overline{\mathcal O}_{\min}}
\newcommand{\aff}{\mathrm{aff}}

\begin{document}
\setlength{\droptitle}{-6em}
\maketitle
\vspace{-4em}

\begin{abstract}
Let $\Lieg$ be a complex simple Lie algebra of type $G_2$, $F_4$, or $E_8$. We prove that its minimal nilpotent orbit closure is not isomorphic to $(T^*X)^\aff$ for any smooth quasi-affine variety $X$. We do not assume equivariance or Poisson compatibility. We also prove the same statement in types $A_1$ and $A_2$.
\end{abstract}

\section{Introduction.}

Let $\Lieg$ be a complex simple Lie algebra, and let $G$ be its adjoint group. Let $x\in\Lieg$. Its adjoint orbit is
\[
    G\cdot x=\{\Ad(g)x:g\in G\}.
\]
The linear map $\ad x:\Lieg\to\Lieg$ is given by $(\ad x)(y)=[x,y]$. We call $x$ {nilpotent} if $\ad x$ is nilpotent. In this case, $G\cdot x$ is called a {nilpotent orbit}. There is a unique nonzero nilpotent orbit of smallest dimension. It is the \emph{minimal nilpotent orbit}, denoted by $\Omin(\Lieg)$, and its Zariksi closure
    $$\Ominbar(\Lieg)=\Omin(\Lieg)\sqcup\{0\}.$$

Choose a Cartan subalgebra $\h\subset\Lieg$ and a Borel subalgebra $\mathfrak b\supset\h$. The Cartan subalgebra determines a root system $\Phi\subset\h^*$, and the Borel subalgebra determines a set of positive roots $\Phi^+$. Let $\theta$ be the highest root, and choose a nonzero vector $e_\theta$ in its root space $\Lieg_\theta$. Then
\[
    \Omin(\Lieg)=G\cdot e_\theta.
\]

For $\Lieg=\mathfrak{sl}_2$, we have
\[
    \Ominbar(\mathfrak{sl}_2)
    =\left\{
      \begin{pmatrix}a&b\\ c&-a\end{pmatrix}:a,b,c\in\C,\ a^2+bc=0
     \right\}.
\]
This is the usual nilpotent cone in $\mathfrak{sl}_2$. For $\Lieg=\mathfrak{sl}_3$, we have
\[
    \Ominbar(\mathfrak{sl}_3)
    =\{A\in\mathfrak{sl}_3:\rk A\le1\}.
\]

The {affinization} of a variety $Y$ is
\[
    Y^\aff\coloneqq\Spec\Gamma(Y,\OO_Y),
\]
where $\Gamma(Y,\OO_Y)$ is the ring of regular functions on $Y$. If this ring is finitely generated, then $Y^\aff$ is an affine variety.
A variety is {quasi-affine} if it is an open subvariety of an affine variety. If $X$ is smooth and quasi-affine, then the total space $T^*X$ of its cotangent bundle is also smooth and quasi-affine.

\begin{samepage}
Our question is the following.
\begin{center}
{\setlength{\fboxsep}{8pt}
\fbox{\begin{minipage}{0.86\textwidth}
\centering
For which complex simple Lie algebras $\Lieg$ can one write $\Ominbar(\Lieg)\cong (T^*X)^\aff$ for some smooth quasi-affine variety $X$?
\end{minipage}}}
\end{center}
\end{samepage}

The answer is negative in types $A_1$ and $A_2$, i.e. neither $\Ominbar(\mathfrak{sl}_2)$ nor $\Ominbar(\mathfrak{sl}_3)$ is isomorphic to $(T^*X)^\aff$ for any smooth quasi-affine variety $X$. We prove these cases in Theorem~\ref{smalltypes}.

Positive examples also exist. Let $U$ be a maximal unipotent subgroup of $\SL_3$. In \cite{JiaSLU}, we constructed an isomorphism
\[
    \bigl(T^*(\SL_3/U)\bigr)^\aff
    \cong\Ominbar(\so_8).
\]
This gives a positive example in type $D_4$.

In \cite[Theorem~1.4]{FuLiu}, Fu and Liu proved a long list of positive answers in types $A_n$ for $n\ge3$, $B_n$ for $n\ge3$, $C_n$ for $n\ge2$, $D_n$ for $n\ge4$, $E_6$, and $E_7$.

In the remaining types $G_2$, $F_4$, $E_8$, our main theorem gives a negative answer.

\begin{theorem}\label{main}
    For $\Lieg$ of type $G_2$, $F_4$, or $E_8$, the closure $\Ominbar(\Lieg)$ of the minimal nilpotent orbit is not isomorphic to $(T^*X)^\aff$ for any smooth quasi-affine variety $X$.
\end{theorem}

\vspace{1em}
\textbf{Acknowledgments.}
The author was supported by NSFC Grant No.~12225108 and the Shuimu Scholar Program in Tsinghua University. The author used ChatGPT extensively during the preparation of this work, and he would like to thank the model GPT-5.6 Sol for pointing out the non-existence result of type $E8$.

\vspace{1em}
\section{The projective orbit and its stabilizer.}\label{projectiveorbit}

Throughout this section and Section~\ref{fixedbound}, let $\Lieg$ be of type $A_2$, $G_2$, $F_4$, or $E_8$.

By the Jacobson-Morozov theorem, we could choose $h_\theta\in\h$ and $f_\theta\in\Lieg_{-\theta}$ such that $(e_\theta,h_\theta,f_\theta)$ forms an $\mathfrak{sl}_2$-triple.  For $-2\le j\le2$, put
$
    \Lieg_j=\{x\in\Lieg:[h_\theta,x]=jx\}.
$
Then
\[
    \Lieg=\Lieg_{-2}\oplus\Lieg_{-1}\oplus\Lieg_0
    \oplus\Lieg_1\oplus\Lieg_2.
\]
Let $P_\theta\subset G$ be the stabilizer of $[e_\theta]\in\mathbb{P}(\Lieg)$. Its Lie algebra is
\[
    \mathfrak p_\theta:=\operatorname{Lie}(P_\theta)
    =\{x\in\Lieg:[x,e_\theta]\in\C e_\theta\}.
\]
The spaces $\Lieg_{\pm2}$ are one-dimensional, and the $\mathfrak{sl}_2$-relations give an isomorphism \[\ad e_\theta:\Lieg_{-1}\to\Lieg_1.\] Hence
\[
    \mathfrak p_\theta=\Lieg_0\oplus\Lieg_1\oplus\Lieg_2,
    \qquad
    \dim G/P_\theta=1+\dim\Lieg_{-1}.
\]
Let $\gamma_\theta:\Gm\to G$ be the one-parameter subgroup such that $d\gamma_\theta(1)=h_\theta$.
Then $$\Ad(\gamma_\theta(t))e_\theta=t^2e_\theta,$$ and the map $G\cdot e_\theta\to G\cdot[e_\theta]$ has one-dimensional fibers. Thus
\[
    \dim\Omin(\Lieg)=2+\dim\Lieg_{-1}.
\]

Let $B\subset G$ be the Borel subgroup determined by the positive roots. Using the Bourbaki labeling \cite[Planches I and VII--IX]{Bourbaki}, let $P_i$ be the maximal parabolic corresponding to $\alpha_i$. Then we have
\[
    \begin{array}{c|c|c}
    \text{type of }\Lieg&P_\theta&\dim\Omin(\Lieg)\\ \hline
    A_2&B&4\\
    G_2&P_2&6\\
    F_4&P_1&16\\
    E_8&P_8&58
    \end{array}.
\]
The adjoint representation is irreducible, so this cone spans $\Lieg$.

\begin{definition}
  We define the \emph{projective stabilizer} of $G/P_\theta\subset\PP(\Lieg)$ as 
    \[
        \Stab_{\PGL(\Lieg)}(G/P_\theta)
        =\{a\in\PGL(\Lieg):a(G/P_\theta)=G/P_\theta\}.
    \]
    Restriction gives a homomorphism from the projective stabilizer to the group of algebraic automorphisms $\Aut(G/P_\theta)$.
    For either algebraic group, we use the superscript ${}^\circ$ to denote its identity component.
\end{definition}

\begin{proposition}\label{projectivestabilizer}
    Regard $G$ as a subgroup of $\PGL(\Lieg)$ via the adjoint representation. Then $$\Stab_{\PGL(\Lieg)}(G/P_\theta)^\circ=G.$$
\end{proposition}
\begin{proof}

    We first show that the restriction homomorphism $$\res:\Stab_{\PGL(\Lieg)}(G/P_\theta)^\circ\rightarrow\Aut(G/P_\theta)$$ is injective. Let $[A]\in\ker(\res)$, with representative $A\in\GL(\Lieg)$. Then every point of $G/P_\theta$ is a one-dimensional eigensubspace of $A$. Let $\lambda_1,\cdots,\lambda_k$ be all the eigenvalues of $A$, and $V_{\lambda_1},\cdots,V_{\lambda_k}$ be the corresponding eigen-subspaces. Then we have $$G/P_\theta\subset\PP(V_{\lambda_1})\cup\cdots\cup\PP(V_{\lambda_k}).$$
    Since $G/P_\theta$ is irreducible, it lies in one of these projective spaces. Since $G/P_\theta$ spans $\PP(\Lieg)$, the corresponding eigenspace is $\Lieg$. Thus $A$ is scalar, and the homomorphism is injective.

    Demazure's theorem  \cite[pp.~181--182, Th\'eor\`eme~1]{Demazure} gives $\Aut(G/P_\theta)^\circ=G$, except in the cases $(C_n,P_1)$, $(B_n,P_n)$, and $(G_2,P_1)$. But none of the four cases of $\PP(\Ominbar)$ above lies in these exceptions.
So
    \[
        G\subseteq
        \Stab_{\PGL(\Lieg)}(G/P_\theta)^\circ
        \hookrightarrow\Aut(G/P_\theta)^\circ=G,
    \]
    and the first inclusion is an equality. This proves the proposition.
\end{proof}

\section{A bound for fixed subspaces.}\label{fixedbound}

Adopt the Bourbaki labeling $\alpha_1,\ldots,\alpha_r$ of the simple roots. Write the highest root $$\theta=\sum_i m_i\alpha_i.$$ For $1\le i\le r$, set
\[
    F_i=\left\{\beta=\sum_{j=1}^r b_j\alpha_j\in\Phi:b_i=m_i\right\}.
\]
Intuitively, they are collections of ``high'' roots.
\begin{proposition}\label{rootcounts}
    For root systems of type $A_2,G_2,F_4,E_8$, respectively, $$\max_i|F_i|=2,2,7,14.$$
\end{proposition}
\begin{proof}
    For each type, the highest root is
    \[
    \begin{aligned}
        \theta^{A_2}
            &=\alpha_1^{A_2}+\alpha_2^{A_2},\\
        \theta^{G_2}
            &=3\alpha_1^{G_2}+2\alpha_2^{G_2},\\
        \theta^{F_4}
            &=2\alpha_1^{F_4}+3\alpha_2^{F_4}
              +4\alpha_3^{F_4}+2\alpha_4^{F_4},\\
        \theta^{E_8}
            &=2\alpha_1^{E_8}+3\alpha_2^{E_8}
              +4\alpha_3^{E_8}+6\alpha_4^{E_8}\\
            &\quad+5\alpha_5^{E_8}+4\alpha_6^{E_8}
              +3\alpha_7^{E_8}+2\alpha_8^{E_8}.
    \end{aligned}
    \]
    For each type, write a root $\beta=\sum_j b_j\alpha_j$ as its coefficient row $(b_1,\ldots,b_r)$. The following tables list the roots in $\bigcup_iF_i$ and the indices $i$ for which $\beta\in F_i$.
    {\small
    \[
    \begin{array}{c|c}
    \multicolumn{2}{c}{A_2}\\
    (b_1,b_2)&\{i:\beta\in F_i\}\\ \hline
    (1,1)&1,2\\
    (1,0)&1\\
    (0,1)&2\\[3pt]
    \multicolumn{2}{c}{G_2}\\
    (b_1,b_2)&\{i:\beta\in F_i\}\\ \hline
    (3,2)&1,2\\
    (3,1)&1\\[3pt]
    \multicolumn{2}{c}{F_4}\\
    (b_1,b_2,b_3,b_4)&\{i:\beta\in F_i\}\\ \hline
    (2,3,4,2)&1,2,3,4\\
    (1,3,4,2)&2,3,4\\
    (1,2,4,2)&3,4\\
    (1,2,3,2)&4\\
    (1,2,2,2)&4\\
    (1,1,2,2)&4\\
    (0,1,2,2)&4
    \end{array}
    \qquad
    \begin{array}{c|c}
    \multicolumn{2}{c}{E_8}\\
    (b_1,\ldots,b_8)&\{i:\beta\in F_i\}\\ \hline
    (2,3,4,6,5,4,3,2)&1,2,3,4,5,6,7,8\\
    (2,3,4,6,5,4,3,1)&1,2,3,4,5,6,7\\
    (2,3,4,6,5,4,2,1)&1,2,3,4,5,6\\
    (2,3,4,6,5,3,2,1)&1,2,3,4,5\\
    (2,3,4,6,4,3,2,1)&1,2,3,4\\
    (2,3,4,5,4,3,2,1)&1,2,3\\
    (2,3,3,5,4,3,2,1)&1,2\\
    (2,2,4,5,4,3,2,1)&1,3\\
    (2,2,3,5,4,3,2,1)&1\\
    (1,3,3,5,4,3,2,1)&2\\
    (2,2,3,4,4,3,2,1)&1\\
    (2,2,3,4,3,3,2,1)&1\\
    (2,2,3,4,3,2,2,1)&1\\
    (2,2,3,4,3,2,1,1)&1\\
    (2,2,3,4,3,2,1,0)&1
    \end{array}
    \]
    }
    Counting the occurrences of each index gives
    \[
        (|F_i|)=(2,2),\qquad(2,1),\qquad(1,2,3,7),\qquad
        (14,8,7,5,4,3,2,1),
    \]
    in types $A_2$, $G_2$, $F_4$, and $E_8$, respectively. The largest entry in each sequence proves the proposition.
\end{proof}

\begin{definition}
    A representation $\rho:\Gm\to\GL(V)$ is \emph{one-sided} if its weights are all nonnegative or all nonpositive. And its fixed subspace is $$V^{\rho(\Gm)}\coloneqq\{v\in V:\rho(t)v=v,\forall t\in\Gm\}.$$
\end{definition}

\begin{lemma}\label{onesided}
    Let $\rho:\Gm\to\GL(\Lieg)$ be a nontrivial one-sided representation preserving $\Ominbar(\Lieg)$. Then $\dim\Lieg^{\rho(\Gm)}\le\max_i|F_i|$.
\end{lemma}
\begin{proof}
    Replace $\rho(t)$ by $\rho(t^{-1})$ if necessary, so that all weights are nonnegative. Let $\bar\rho:\Gm\to\PGL(\Lieg)$ be the homomorphism induced by $\rho$. Its image is connected and preserves $G/P_\theta$. Hence
    \[
        \bar\rho(\Gm)\subset
        \Stab_{\PGL(\Lieg)}(G/P_\theta)^\circ
    \]
    By Proposition~\ref{projectivestabilizer}, $\bar\rho$ is a cocharacter of $G$. After conjugating $\rho$ by an element of $G$, we may assume that $\bar\rho$ is anti-dominant. This does not change the weights of $\rho$ or the dimension of its fixed subspace. Thus
    \[
        \bar\rho(t)=\Ad(\mu(t)^{-1})
    \]
    for a dominant cocharacter $\mu:\Gm\to G$. Let $\omega_1^\vee,\ldots,\omega_r^\vee$ be the fundamental coweights. Since $G$ is adjoint, they are cocharacters of $G$.
    Since $\rho(t)\Ad(\mu(t))$ is scalar and depends homomorphically on $t$, these scalars form a character of $\Gm$. Thus, for some $c\in\mathbb Z$,
    \[
        \rho(t)=t^c\Ad(\mu(t)^{-1}),\qquad
        \mu=\sum_i a_i\omega_i^\vee,\quad a_i\in\mathbb Z_{\ge0}.
    \]
    The $\rho$-weight on $\h$ is $c$, and its weight on $\Lieg_\alpha$ is $c-\langle\alpha,\mu\rangle$. Thus $c\ge0$. If $c=0$, nonnegativity on $\Lieg_\alpha$ and $\Lieg_{-\alpha}$ gives $\langle\alpha,\mu\rangle=0$ for every root $\alpha$. Then $\mu=0$ and $\rho$ is trivial. 
    Therefore $c>0$. The Cartan subalgebra has no fixed vector, and the fixed subspace is the sum of the root spaces for which $\langle\alpha,\mu\rangle=c$.

    If the fixed subspace is zero, there is nothing to prove. Suppose $\Lieg^{\rho(\Gm)}$ is nonzero and contains a root space $\Lieg_\beta$. Then $\mu\neq 0$, and we choose $i$ with $a_i>0$. Write $$\beta=\sum_j b_j\alpha_j.$$ Then $\langle\beta,\mu\rangle=c>0$, so $\beta$ is positive. Since $\mu$ is dominant, $\langle\theta,\mu\rangle\ge\langle\beta,\mu\rangle=c$. The weight on $\Lieg_\theta$ is nonnegative, so $\langle\theta,\mu\rangle\le c$. Therefore
    \[
        0=\langle\theta-\beta,\mu\rangle
        =\sum_j a_j(m_j-b_j).
    \]
    Every summand is nonnegative because $\theta-\beta$ is a nonnegative integral combination of simple roots. Since $a_i>0$, we have $b_i=m_i$. Thus every fixed root lies in $F_i$. The root spaces are one-dimensional, so the fixed subspace has dimension at most $|F_i|\le\max_j|F_j|$.
\end{proof}

\section{Quasi-affine varieties and affinizations.}\label{quasiaffinesection}

All results in this section are well known. We include them here only for completeness.

\begin{definition}[Quasi-affine variety]\label{quasiaffinedefinition}
    A variety $X$ is \emph{quasi-affine} if it is isomorphic to an open subvariety of an affine variety \cite[Definition~28.19.1]{Stacks}.
\end{definition}

\begin{definition}[Affinization]\label{affinizationdefinition}
    Let $X$ be a variety. Set
    \[
        A_X=\Gamma(X,\OO_X),
        \qquad
        X^\aff=\Spec A_X.
    \]
    By \cite[Lemma~26.6.4]{Stacks}, the identity homomorphism of $A_X$ defines a canonical morphism
    \[
        \iota_X:X\longrightarrow X^\aff.
    \]
    We call the affine scheme $X^\aff$ the \emph{affinization} of $X$. It need not be a variety because $A_X$ need not be finitely generated.
\end{definition}

\begin{proposition}\label{affinizationmap}
    Let $X$ be a quasi-affine variety. Then the canonical map $\iota_X:X\to X^\aff$ is an open immersion.
\end{proposition}
\begin{proof}
    This is \cite[Lemma~28.19.4]{Stacks}.
\end{proof}

\begin{proposition}\label{vectorbundlequasiaffine}
    Let $p:E\to X$ be a vector bundle of finite rank over a quasi-affine variety. Then $E$ is quasi-affine. If $X$ is smooth, then $E$ is smooth.
\end{proposition}
\begin{proof}
    The projection $p$ is affine by \cite[Definition~27.6.2]{Stacks}. Thus $p$ is quasi-affine by \cite[Definition~29.13.1]{Stacks}. The structure morphism $X\to\Spec\C$ is quasi-affine. Their composition is quasi-affine by \cite[Lemma~29.13.4]{Stacks}. Thus $E$ is quasi-affine.

    Suppose that $X$ is smooth. Every local trivialization identifies $p^{-1}(U)$ with $U\times\mathbb A^r$ for a smooth open subvariety $U\subset X$. Thus $E$ is smooth.
\end{proof}

\begin{lemma}\label{cotangentgrading}
    Let $X$ be a smooth quasi-affine variety. Then $T^*X$ is smooth and quasi-affine. Fiber dilation gives the weight decomposition
    \[
        \Gamma(T^*X,\OO)
        =\bigoplus_{k\ge0}H^0(X,\Sym^kT_X).
    \]
    If $\Gamma(T^*X,\OO)$ is finitely generated, then $\Gamma(X,\OO_X)$ is finitely generated.
\end{lemma}
\begin{proof}
    Proposition~\ref{vectorbundlequasiaffine} shows that $T^*X$ is smooth and quasi-affine. Let $p:T^*X\to X$ be the projection. Let $\Gm$ act on $T^*X$ by fiber dilation,
    $t\cdot(x,\xi)=(x,t\xi)$. By \cite[Definition~27.6.1]{Stacks} and \cite[Lemma~27.4.6]{Stacks},
    \[
        p_*\OO_{T^*X}=\Sym T_X=\bigoplus_{k\ge0}\Sym^kT_X.
    \]
    The summand $\Sym^kT_X$ is the weight-$k$ part for fiber dilation. Since $X$ is quasi-compact and separated, global sections commute with this direct sum by \cite[Lemma~30.6.1]{Stacks}. Thus
    \[
        \Gamma(T^*X,\OO)
        =\bigoplus_{k\ge0}H^0(X,\Sym^kT_X).
    \]

    Suppose that $B=\Gamma(T^*X,\OO)$ is finitely generated. The fiber-dilation action makes $B$ a finitely generated $\Gm$-algebra. Since $\Gm$ is reductive, Hilbert's finite-generation theorem shows that $B^{\Gm}$ is finitely generated. The weight decomposition gives
    \[
        B^{\Gm}=B_0=H^0(X,\OO_X)=\Gamma(X,\OO_X),
    \]
    which proves the last assertion.
\end{proof}

\begin{lemma}\label{smallboundary}
    Let $X$ be a smooth quasi-affine variety. If $\Gamma(X,\OO_X)$ is finitely generated, then the canonical map $X\to X^\aff$ is an open immersion and
    \[
        \operatorname{codim}_{X^\aff}(X^\aff\setminus X)\ge2.
    \]
\end{lemma}
\begin{proof}
    By Proposition~\ref{affinizationmap}, the canonical map identifies $X$ with an open subvariety of $Z=X^\aff$. Since $X$ is smooth, it is normal by \cite[Lemma~33.25.4]{Stacks}. Hence $A=\Gamma(X,\OO_X)$ is a normal domain by \cite[Lemma~28.7.9]{Stacks}, so $Z=\Spec A$ is an irreducible normal affine variety. By construction, $X$ and $Z$ have the same ring of regular functions. The codimension statement now follows from \cite[Theorem~4.2]{Grosshans}.
\end{proof}

\section{The tangent representation $\rho$.}\label{tangentsection}

Let $Y$ be an affine variety (hence irreducible). Let $\Gm$ act on $Y$. Let $y_0\in Y$ be a fixed point. Set $A=\C[Y]$, and let $\mathfrak m\subset A$ be the ideal of $y_0$. Write $\gr_{\mathfrak m}A=\bigoplus_{q\ge0}\mathfrak m^q/\mathfrak m^{q+1}$. The \emph{tangent cone} of $Y$ at $y_0$ is $C_{y_0}Y=\Spec\gr_{\mathfrak m}A$, and $T_{y_0}Y=(\mathfrak m/\mathfrak m^2)^*$.

For $t\in\Gm$, set $t^*f(y)=f(t\cdot y)$. Write
\[
    A=\bigoplus_{n\in\mathbb Z}A_n,
    \qquad
    A_n=\{f\in A:t^*f=t^nf\text{ for every }t\in\Gm\}.
\]
Let $\rho$ be the \emph{tangent representation} at $y_0$ of the $\Gm$-action on $Y$, which is the homomorphism
$\rho:\Gm \longrightarrow \GL(T_{y_0}Y)$ with $\rho(t)$ the differential of $z\mapsto t\cdot z$ at $y_0$, i.e.
    \[
        \begin{aligned}
            \rho(t):T_{y_0}Y& \longrightarrow T_{y_0}Y\\
            v &\longmapsto
            \left(\bar f\mapsto v\bigl(\overline{t^*f}\bigr)\right),
        \end{aligned}
    \]
    where $\bar f\in T_{y_0}^*Y=\mathfrak m/\mathfrak m^2$ is the class of $f\in\mathfrak m$.

\begin{lemma}\label{weightsigns}
    Let $\lambda\in\mathbb Z\setminus\{0\}$. If $A_\lambda\ne0$, then $\rho$ has a weight of the same sign as $\lambda$. In particular, if the $\Gm$-action is nontrivial, then $\rho$ is nontrivial.
\end{lemma}
\begin{proof}
    Choose $0\ne f\in A_\lambda$. Since $y_0$ is fixed, $f\in\mathfrak m$.
    The ring $A$ is a domain, so $A\to A_{\mathfrak m}$ is injective. The local Krull intersection theorem gives $\bigcap_{q\ge1}(\mathfrak mA_{\mathfrak m})^q=0$.
    Hence $\bigcap_{q\ge1}\mathfrak m^q=0$ in $A$. Thus $f\in\mathfrak m^q\setminus\mathfrak m^{q+1}$ for some $q\ge1$. Its image in $\mathfrak m^q/\mathfrak m^{q+1}$ is nonzero and has weight $\lambda$.

    Multiplication gives a $\Gm$-equivariant surjection
    \[
        \Sym^q(\mathfrak m/\mathfrak m^2)
        \longrightarrow
        \mathfrak m^q/\mathfrak m^{q+1}.
    \]
    Hence $\lambda$ is a sum of $q$ weights of $\mathfrak m/\mathfrak m^2$. At least one of them has the same sign as $\lambda$. It follows that $\rho$ has a weight of the same sign as $\lambda$.

    If the action is nontrivial, then $A_\lambda\ne0$ for some $\lambda\ne0$. This proves the last assertion.
\end{proof}

\begin{lemma}\label{tangentcone}
    Suppose that $Y\subset V$ is an affine cone which spans $V$, and that $y_0=0$. Then $T_{y_0}Y=V$ and $C_{y_0}Y=Y$ as closed subschemes of $V$. Moreover, $\rho(\Gm)$ preserves $Y$.
\end{lemma}
\begin{proof}
    Let $I\subset\Sym(V^*)$ be the homogeneous ideal of $Y$. Since $Y$ spans $V$, we have $I_1=0$. Thus $T_{y_0}Y=V$. By definition, the ideal of $C_{y_0}Y$ is generated by the lowest-degree homogeneous parts of the elements of $I$. Since $I$ is homogeneous, this ideal is $I$. Hence $C_{y_0}Y=Y$ as closed subschemes of $V$.

    The given action need not be the scalar multiplication of the cone. Pullback preserves $\mathfrak m$ and all its powers. Thus the closed embedding $C_{y_0}Y\subset T_{y_0}Y$ is equivariant. Hence $\rho(\Gm)$ preserves $C_{y_0}Y=Y$.
\end{proof}

\begin{definition}[Fixed-point subscheme]\label{fixedsubscheme}
    Let $I_{\mathrm{fix}}=\langle A_n:n\ne0\rangle\subset A$. The quotient $A/I_{\mathrm{fix}}$ is concentrated in degree zero, and $(I_{\mathrm{fix}})_0=\sum_{n>0}A_nA_{-n}$. The \emph{fixed-point subscheme} is
    \[
        Y^{\Gm}=\Spec(A/I_{\mathrm{fix}})
        \cong
        \Spec\left(A_0\Big/\sum_{n>0}A_nA_{-n}\right).
    \]
\end{definition}

\begin{lemma}\label{fixedpoints}
     If $\rho$ is one-sided, then $Y^{\Gm}\cong\Spec A_0$, and
     \[
        T_{y_0}(Y^{\Gm})=(T_{y_0}Y)^{\rho(\Gm)}.
    \]
\end{lemma}
\begin{proof}
    Replacing the action by its inverse if necessary, assume that the weights of $\rho$ are nonnegative. Lemma~\ref{weightsigns} gives $A_n=0$ for every $n<0$.
    Since the negative-weight spaces vanish, we have
    \[
        I_{\mathrm{fix}}=\bigoplus_{n>0}A_n,
        \qquad
        A/I_{\mathrm{fix}}=A_0=A^{\Gm}.
    \]
    Therefore
    \[
        Y^{\Gm}\cong\Spec A^{\Gm}.
    \]

    Observe that the image of $I_{\mathrm{fix}}$ in $\mathfrak m/\mathfrak m^2$ is exactly its positive-weight part:
    \[
        \frac{I_{\mathrm{fix}}+\mathfrak m^2}{\mathfrak m^2}
        =(\mathfrak m/\mathfrak m^2)_{>0}.
    \]
   The maximal ideal of $y_0$ in $A/I_{\mathrm{fix}}$ is $\mathfrak m/I_{\mathrm{fix}}$. Hence
    \[
        T^*_{y_0}(Y^{\Gm})=\frac{\mathfrak m/I_{\mathrm{fix}}}{(\mathfrak m/I_{\mathrm{fix}})^2}=\frac{\mathfrak m}{\mathfrak m^2+I_{\mathrm{fix}}}
        =(\mathfrak m/\mathfrak m^2)_0.
    \]
    Taking duals and using the weight decomposition gives
    \[
        T_{y_0}(Y^{\Gm})
        =\bigl((\mathfrak m/\mathfrak m^2)_0\bigr)^*
        =(T_{y_0}Y)_0
        =(T_{y_0}Y)^{\rho(\Gm)}.\qedhere
    \]
\end{proof}

\begin{proof}[Proof of Theorem~\ref{main}]
    Suppose that $(T^*X)^\aff\cong\Ominbar(\Lieg)$. The assumed isomorphism makes $\Gamma(T^*X,\OO)$ finitely generated. Lemma~\ref{cotangentgrading} shows that $T^*X$ is smooth and quasi-affine. By Lemma~\ref{smallboundary}, $T^*X$ is an open subvariety of $(T^*X)^\aff$. Hence
    \[
        \dim X=\frac12\dim\Ominbar(\Lieg)>0.
    \]

    Let $\Gm$ act on $T^*X$ by $t\cdot(x,\xi)=(x,t\xi)$. This action induces an algebraic $\Gm$-action on $(T^*X)^\aff$. Lemma~\ref{cotangentgrading} gives the grading
    \[
        \Gamma(T^*X,\OO)
        =\bigoplus_{k\ge0}H^0(X,\Sym^kT_X).
    \]
    Under the assumed isomorphism, this gives a $\Gm$-action on $\Ominbar(\Lieg)$. This action is nontrivial. Let $y_0$ be the vertex. It is the unique singular point, so it is fixed. Let $\rho$ be the tangent representation at $y_0$. By Lemma~\ref{weightsigns}, $\rho$ is nontrivial. The fiber grading has only nonnegative degrees. Hence pullback on the cotangent space at $y_0$, and therefore $\rho$, has only nonnegative weights. Lemma~\ref{tangentcone} gives $T_{y_0}\Ominbar(\Lieg)=\Lieg$. It also shows that $\rho(\Gm)$ preserves $\Ominbar(\Lieg)$.

    Lemma~\ref{fixedpoints} gives
    \[
        \left((T^*X)^\aff\right)^{\Gm}
        \cong\Spec\Gamma(T^*X,\OO)_0
        =\Spec\Gamma(X,\OO_X)
        =X^\aff.
    \]
    Let $x_0\in X^\aff$ correspond to $y_0$. Then Lemma~\ref{fixedpoints} gives
    \[
        \Lieg^{\rho(\Gm)}=T_{x_0}(X^\aff).
    \]
    Lemma~\ref{cotangentgrading} shows that $\Gamma(X,\OO_X)$ is finitely generated. By Lemma~\ref{smallboundary}, the canonical map identifies $X$ with an open subvariety of $X^\aff$.

    The variety $X^\aff$ is irreducible and contains $X$ as a dense open subvariety. Hence $\dim X^\aff=\dim X$. By Lemma~\ref{onesided},
    \[
        \dim X=\dim X^\aff
        \le\dim T_{x_0}(X^\aff)
        =\dim\Lieg^{\rho(\Gm)}\le\max_i|F_i|.
    \]
    Proposition~\ref{rootcounts} gives the upper bounds $2$, $7$, and $14$ in types $G_2$, $F_4$, and $E_8$. The table in Section~\ref{projectiveorbit} gives the orbit dimensions $6$, $16$, and $58$, so $\dim X=3$, $8$, and $29$, respectively. This is a contradiction.
\end{proof}

\begin{theorem}\label{smalltypes}
    For $n=2,3$, $\Ominbar(\mathfrak{sl}_n)$ is not isomorphic to $(T^*X)^\aff$ for any smooth quasi-affine variety $X$.
\end{theorem}
\begin{proof}
    Suppose that there is an isomorphism $\varphi:(T^*X)^\aff\xrightarrow{\sim}\Ominbar(\mathfrak{sl}_n)$. This makes $\Gamma(T^*X,\OO)$ finitely generated. Lemma~\ref{cotangentgrading} shows that $T^*X$ is smooth and quasi-affine. By Lemma~\ref{smallboundary}, $T^*X$ is an open subvariety of $(T^*X)^\aff$. Hence $$\dim X=(1/2)\dim \Ominbar(\mathfrak{sl}_n)=n-1.$$

    Lemma~\ref{cotangentgrading} identifies $\Gamma(X,\OO_X)$ with the degree-zero part of $\Gamma(T^*X,\OO)$ and shows that this ring is finitely generated. Lemma~\ref{smallboundary} identifies $X$ with an open subvariety of $X^\aff$ whose boundary has codimension at least two.

    Suppose first that $n=2$. Then $X^\aff$ is an irreducible curve, so the boundary is empty. Thus $X=X^\aff$ is affine. Hence $(T^*X)^\aff=T^*X$ is smooth, but $\Ominbar(\mathfrak{sl}_2)$ is singular.

    Suppose that $n=3$. Via $\varphi$, the action $t\cdot(x,\xi)=(x,t\xi)$ gives a $\Gm$-action on $\Ominbar(\mathfrak{sl}_3)$. Since $\dim X=2$ and $T^*X\subset(T^*X)^\aff$ is an open subvariety, this action is nontrivial. Let $y_0$ be the vertex. It is the unique singular point, so it is fixed. Let $\rho$ be the tangent representation at $y_0$. By Lemma~\ref{weightsigns}, $\rho$ is nontrivial. The fiber grading has only nonnegative degrees. Hence $\rho$ is one-sided. By Lemma~\ref{tangentcone}, we have $T_{y_0}\Ominbar(\mathfrak{sl}_3)=\mathfrak{sl}_3$, and $\rho(\Gm)$ preserves $\Ominbar(\mathfrak{sl}_3)$.

    By Lemma~\ref{fixedpoints}, the fixed-point scheme is isomorphic to $X^\aff$. Let $x_0\in X^\aff\subset(T^*X)^\aff$ be the point such that $\varphi(x_0)=y_0$. By the tangent-space assertion of the same lemma, we have the following equality of vector spaces:
    \[
        \mathfrak{sl}_3^{\rho(\Gm)}=T_{x_0}(X^\aff).
    \]
    The irreducible affine variety $X^\aff$ has dimension $2$. By Lemma~\ref{onesided} and Proposition~\ref{rootcounts}, we have the following inequalities:
    \[
        2=\dim X^\aff\le\dim T_{x_0}(X^\aff)
        =\dim\mathfrak{sl}_3^{\rho(\Gm)}
        \le\max_i|F_i|=2.
    \]
    Thus $\dim T_{x_0}(X^\aff)=\dim X^\aff=2$. Since $X^\aff$ is irreducible, $x_0$ is a smooth point of $X^\aff$.

    The boundary $X^\aff\setminus X$ has codimension at least two in the noetherian surface $X^\aff$. Hence it is zero-dimensional and finite. Set $X'=X\cup\{x_0\}$. The complement $X^\aff\setminus X'=(X^\aff\setminus X)\setminus\{x_0\}$ is finite and closed. Thus $X'$ is open in the affine variety $X^\aff$, so it is quasi-affine. It is smooth because $X$ is smooth and $x_0$ is a smooth point of $X^\aff$.

    Let $z_0\in T^*_{x_0}X'$ be the zero covector. The complement $T^*X'\setminus T^*X$ is contained in the rank-two fiber $T^*_{x_0}X'$. Hence it has codimension at least two in $T^*X'$. The smooth variety $T^*X'$ is normal, so Hartogs extension gives
    \[
        \Gamma(T^*X',\OO)=\Gamma(T^*X,\OO).
    \]
    Let $\beta:(T^*X')^\aff\xrightarrow{\sim}(T^*X)^\aff$ be the isomorphism induced by the inverse of the restriction map. Then $\varphi\circ\beta$ identifies $(T^*X')^\aff$ with $\Ominbar(\mathfrak{sl}_3)$.
    In particular, $\Gamma(T^*X',\OO)$ is finitely generated. Lemma~\ref{cotangentgrading} applies to $X'$ and shows that $T^*X'$ is smooth and quasi-affine. By Lemma~\ref{smallboundary}, the canonical map $T^*X'\to (T^*X')^\aff$ is an open immersion.

    Let $\alpha:T^*X'\to(T^*X')^\aff$ be this canonical open immersion. The complement $X'\setminus X$ has codimension at least two in the smooth surface $X'$. Hartogs extension gives $\Gamma(X',\OO_{X'})=\Gamma(X,\OO_X)$. Functions of positive fiber degree vanish on the zero section. Thus the restriction of $\beta\circ\alpha$ to the zero section is the inclusion $X'\subset X^\aff\subset(T^*X)^\aff$. Hence $\beta\circ\alpha$ sends $z_0$ to $x_0$, and $\varphi$ sends $x_0$ to $y_0$. Therefore $$\varphi\circ\beta\circ\alpha(z_0)=y_0$$  Since $\alpha$ is an open immersion and $\beta$ and $\varphi$ are isomorphisms, so $\varphi\circ\beta\circ\alpha$ should send smooth points to smooth points. But but $y_0$ is a singular point of $\Ominbar(\mathfrak{sl}_3)$, contradiction.
\end{proof}

\medskip
\footnotesize
\noindent Boming Jia, \textit{Email}: \texttt{jiabm@tsinghua.edu.cn}\\
\textsc{Yau Mathematical Sciences Center,\\
Jingzhai 301, Tsinghua University,\\
Beijing, 100084, China.}


{\small
\begin{thebibliography}{Jia21}
\setlength{\itemsep}{0pt}
\setlength{\parsep}{0pt}

\bibitem[Bou81]{Bourbaki}
N.~Bourbaki, \emph{Groupes et alg\`ebres de Lie, Chapitres 4--6}, Masson, Paris, 1981.

\bibitem[Dem77]{Demazure}
M.~Demazure, Automorphismes et d\'eformations des vari\'et\'es de Borel, \emph{Invent. Math.} \textbf{39} (1977), no.~2, 179--186, doi:10.1007/BF01390108.

\bibitem[FL25]{FuLiu}
B.~Fu and J.~Liu, The affine closure of cotangent bundles of horospherical spaces, arXiv:2502.06383v3 [math.AG] (2026).

\bibitem[Gro97]{Grosshans}
F.~D.~Grosshans, \emph{Algebraic homogeneous spaces and invariant theory}, Lecture Notes in Mathematics, vol.~1673, Springer-Verlag, Berlin, 1997, doi:10.1007/BFb0093525.

\bibitem[Jia21]{JiaSLU}
B.~Jia, The affine closure of $T^*(\SL_n/U)$, arXiv:2112.08649v1 [math.AG] (2021).

\bibitem[Sta]{Stacks}
The Stacks Project Authors, \emph{The Stacks Project},
\url{https://stacks.math.columbia.edu}.

\end{thebibliography}
}
\end{document}